\documentclass[letterpaper, 10pt, conference]{ieeeconf}

\usepackage{cite}
\usepackage{amsmath,amssymb,amsfonts}
\usepackage{algorithmic}
\usepackage{graphicx}
\usepackage{textcomp}
\usepackage{epstopdf}
\usepackage{mathbbol}
\usepackage{mathrsfs}
\usepackage{url}
\usepackage{amsmath,amsfonts, amssymb}
\usepackage{cite}

\newtheorem{theorem}{Theorem}[section]
\newtheorem{lemma}[theorem]{Lemma}
\newtheorem{remark}[theorem]{Remark}

\newtheorem{assumption}[theorem]{Assumption}

\makeatletter
\newcommand{\removelatexerror}{\let\@latex@error\@gobble}
\makeatother

\usepackage[T1]{fontenc}
\usepackage[utf8]{inputenc}
\usepackage[english]{babel}
\usepackage[usenames,dvipsnames,table]{xcolor}
\usepackage{amsmath,amssymb}
\usepackage{graphicx}
\usepackage{url}
\usepackage{float}
\usepackage{tikz}
\usetikzlibrary{shapes,arrows,positioning}
\usepackage{empheq}

\usepackage{pgfplots} 		       
\usepackage{dsfont}   		       
 \usepackage[siunitx]{circuitikz} 
\usepackage[mathscr]{euscript}
\usepackage{mathtools}
\usepackage{xargs} 

\newcommand{\argmin}{\operatorname{argmin}}

\newcommand{\eps}{\epsilon}
\newcommand{\real}{\mathbb{R}}
\newcommand{\realextended}{\overline{\real}}

\newcommand{\naturalnumbers}{\mathbb{N}}
\newcommand{\norm}[1]{\ensuremath{\| #1 \|}}

\newcommand{\map}[3]{#1:#2 \rightarrow #3}
\newcommand{\setdef}[2]{\{#1 \; | \; #2\}}
\newcommand{\setdefb}[2]{\Bigl\{#1 \; \Big| \; #2\Bigr\}}

\newcommand{\MM}{\mathcal{M}}
\newcommand{\abs}[1]{\ensuremath{\left\lvert{#1}\right\rvert}}
\newcommand{\cl}{\operatorname{cl}}

\newcommand{\setr}[1]{\{#1\}}

\newcommand{\FF}{\mathcal{F}}

\newcommand{\xb}{\bar{x}}
\newcommand{\tb}{\bar{t}}

\newcommand{\st}{\operatorname{s.} \operatorname{t.}} 

\newcommand{\Eb}{\mathbb{E}}
\newcommand{\Pb}{\mathbb{P}}
\newcommand{\Rb}{\mathbb{R}}
\newcommand{\Qb}{\mathbb{Q}}
\newcommand{\Nb}{\mathbb{N}}

\newcommand{\BB}{\mathcal{B}}
\newcommand{\HH}{\mathcal{H}}

\newcommand{\YY}{\mathcal{Y}}
\newcommand{\TT}{\mathcal{T}}

\newcommand{\PP}{\mathcal{P}}

\newcommand{\WW}{\mathcal{W}}

\renewcommand{\SS}{\mathcal{S}}

\newcommand{\CVaR}[1]{\operatorname{CVaR}^{#1}}
\newcommand{\VaR}[1]{\operatorname{VaR}^{#1}}

\newcommand{\Pbhat}{\widehat{\Pb}}
\newcommand{\data}{\widehat{\xi}}

\newcommand{\Jrcp}{\mathsf{J}^{\mathtt{RCP}}} 
\newcommand{\Jdrrcp}{\mathsf{J}^{\mathtt{DRRCP}}}
\newcommand{\xdrrcp}{x^{\mathtt{DRRCP}}}

\newcommand{\FFrcp}{\FF^{\mathtt{RCP}}}
\newcommand{\FFdrrcp}{\FF^{\mathtt{DRRCP}}}
\newcommand{\FFrcpo}{\FF^{\mathtt{RCP*}}}
\newcommand{\FFdrrcpo}{\FF^{\mathtt{DRRCP*}}}
\newcommand{\dist}{\operatorname{dist}}
\newcommand{\vccp}{v^{\mathtt{CCP}}}
\newcommand{\vdrccp}{\widehat{v}^{\mathtt{DRCCP}}_N}
\newcommand{\xdrccp}{x^{\mathtt{DRCCP}}}
\newcommand{\Jccp}{\mathsf{J}^{\mathtt{CCP}}} 
\newcommand{\Jdrccp}{\mathsf{J}^{\mathtt{DRCCP}}}

\newcommand{\oprocendsymbol}{\hbox{$\bullet$}}
\newcommand{\oprocend}{\relax\ifmmode\else\unskip\hfill\fi\oprocendsymbol}
\newcommand{\longthmtitle}[1]{\mbox{}\textup{\textsl{(#1):}}}

\renewcommand{\theenumi}{(\roman{enumi})}

\title{Consistency of Distributionally Robust Risk- and Chance-Constrained Optimization under Wasserstein Ambiguity Sets}

\author{Ashish Cherukuri and Ashish R. Hota
\thanks{A. Cherukuri and A. R. Hota are respectively affiliated with the Engineering and Technology Institute Groningen, University of Groningen, The Netherlands (\texttt{a.k.cherukuri@rug.nl}) and the Department of Electrical Engineering, Indian Institute of Technology, Kharagpur, India (\texttt{ahota@ee.iitkgp.ac.in}). This work is supported in part by a grant from IIT Kharagpur under the ISIRD scheme.}}%

\begin{document}

\maketitle
\thispagestyle{empty}

\begin{abstract}
We study stochastic optimization problems with chance and risk constraints, where in the latter, risk is quantified in terms of the conditional value-at-risk (CVaR). We consider the distributionally robust versions of these problems, where the constraints are required to hold for a family of distributions constructed from the observed realizations of the uncertainty via the Wasserstein distance. Our main results establish that if the samples are drawn independently from an underlying distribution and the problems satisfy suitable technical assumptions, then the optimal value and optimizers of the distributionally robust versions of these problems converge to the respective quantities of the original problems, as the sample size increases.
\end{abstract}

\section{Introduction}

Optimization problems under uncertain constraints are pervasive in engineering applications. In the paradigm of {\it chance-constrained programs} (CCPs), uncertain parameters are treated as random variables and the uncertain constraints are required to be satisfied with a high probability. However, the feasibility set of a CCP is in general non-convex \cite{nemirovski2006convex}. Furthermore, although the probability of constraint violation is required to be small, the magnitude of constraint violation could potentially be unbounded which is not desirable. 

Consequently, recent approaches model uncertain constraints via coherent risk measures that preserve analytical tractability; specifically the conditional value-at-risk (CVaR) \cite{rockafellar2000optimization,artzner1999coherent}. In contrast with chance constraints, (i) CVaR preserves the convexity of the feasibility set, (ii) it requires the magnitude of constraint violation to be bounded in expectation (to be made more precise in Section \ref{subsec:cc-approx}), and (iii) CVaR constraints provide a convex inner approximation of chance constraints \cite{nemirovski2006convex}. Accordingly, CVaR-constrained programs (referred to as {\it risk-constrained programs (RCPs)}) have seen widespread applications in financial engineering \cite{krokhmal2002portfolio}, stochastic optimal control \cite{van2016distributionally,PS-MS-PP:19-ecc, singh2018framework}, safety-critical control applications \cite{samuelson2018safety}, robotics \cite{hakobyan2019risk} and energy systems \cite{summers2015stochastic}.

In order to solve stochastic optimization problems in general and CCPs and RCPs in particular, the decision maker needs to know the probability distribution of uncertain parameters. In practice, this information is often unavailable and instead, the decision maker has access to data about the uncertainty in the form of samples. Accordingly, recent work has focused on constructing a family of probability distributions or an {\it ambiguity set} from the observed samples followed by solving the uncertain optimization problem in a worst-case sense for all distributions in the ambiguity set. This approach is referred to as {\it distributionally robust} optimization. Within this paradigm, ambiguity sets defined via the Wasserstein distance (see Section \ref{sec:prelims} for the definition) have been shown to have desirable out-of-sample performance and analytical tractability \cite{gao2016wasserstein,esfahani2018data,hota2018data}. Motivated by these attractive features, several recent works have proposed approximations and finite-dimensional reformulations of Wasserstein distributionally robust chance and CVaR constrained programs \cite{hota2018data,xie2018drccp,kuhn2018drccp,fatma2020drccp}. This class of problems have also been studied in the context of statistical learning \cite{shafieezadeh2019regularization}, data-driven control \cite{yang2017convex,coulson2020distributionally}, and optimal power flow \cite{poolla2020wasserstein}, among others.

Note that the Wasserstein ambiguity set is defined directly in terms of the available samples that are drawn from an underlying data-generating distribution. Consequently, the distributionally robust problem instance is a random instance of the original CCP (or RCP) defined in terms of the underlying distribution. Therefore, in addition to analytical tractability and finite sample guarantees, it is desirable to analyze how well the optimal solution of the (random) distributionally robust program approximates the optimal solution of the original CCP (or RCP); particularly in the regime when the number of samples grows to infinity. This property is termed as {\it asymptotic consistency} in stochastic programming. While {\it asymptotic consistency} has been established for Wasserstein distributionally robust optimization problems \cite{esfahani2018data}, analogous results for chance- and risk-constrained programs have not been explored in the prior work.

In this paper, we show under suitable assumptions that if the samples are being drawn from an underlying distribution $\Pb$, then the optimal solution and optimizers of the distributionally robust CCP or RCP converge to the corresponding quantities of the CCP or RCP (defined with respect to $\Pb$), as the number of samples increases and the size of the ambiguity set shrinks. We show that the convergence of the optimal values is from above if the rate at which the ambiguity set shrinks is chosen carefully. Our results provide the much needed asymptotic theoretical justification for Wasserstein distributionally robust constrained optimization programs. 

\noindent {\bf Notation:} The sets of real, positive real, non-negative real, and natural numbers are denoted by $\Rb$, $\Rb_{>0}$, $\Rb_{\ge 0}$, and $\naturalnumbers$, respectively. The extended reals are $\realextended = \real \cup \{+ \infty, - \infty \}$. For $N \in \naturalnumbers$, we let $[N] := \{1,2,\ldots,N\}$. For brevity, we denote $\max(x,0)$ by $x_+$. The closure of a set $\SS$ is denoted by $\cl(\SS)$. For a set $\SS$ and $N \in \naturalnumbers$, we denote the $N$-fold cartesian product as $\SS^N := \Pi_{i=1}^N \SS$. Similar notation holds for the $N$-fold product of any probability distribution.

\section{Technical Preliminaries}\label{sec:prelims}

Here we formally define the notion of CVaR, Wasserstein distance, and data-driven ambiguity sets. 

\subsubsection{(Conditional) Value-at-Risk}\label{subsec:cc-approx}

Let $Y$ be a (real-valued) random variable with distribution $\Pb$. For a tolerance level $\alpha \in (0,1)$, the value-at-risk (VaR) of $Y$ at level $\alpha$ is  
\begin{equation}\label{eq:def_var}
\VaR{\Pb}_\alpha(Y) := \inf \setdef{y \in \Rb}{\Pb(Y \leq y) \geq 1-\alpha}.
\end{equation}
That is, it is the $(1-\alpha)$-quantile of the distribution of $Y$. The conditional value-at-risk (CVaR) of $Y$ at level $\alpha$ is 
\begin{align}
\CVaR{\Pb}_{\alpha}(Y):= \inf_{t \in \Rb} \, \{
\alpha^{-1} \Eb_{\Pb}[(Y+t)_+] - t\} . \label{eq:CVaR-def}
\end{align}
If $Y$ has a continuous distribution, then $\CVaR{\Pb}_{\alpha}(Y) = \Eb_{\Pb}[Y | Y \geq \VaR{\Pb}_\alpha(Y)]$, i.e., it is the conditional expectation of $Y$ given that $Y$ exceeds $\VaR{\Pb}_\alpha(Y)$. 

\subsubsection{Wasserstein ambiguity sets}\label{subsec:wasserstein}

Assume  $\Xi \subseteq \Rb^m$ and $d$ to be a complete metric on $\Xi$. Let $\BB({\Xi})$ and $\PP(\Xi)$ be the Borel $\sigma$-algebra and the set of Borel probability measures on $\Xi$, resp. Let $\PP_1(\Xi) \subseteq \PP(\Xi)$ be the set of measures with finite first moment. Following \cite{esfahani2018data}, the $1$-Wasserstein distance between any two measures $\mu, \nu \in \PP_1(\Xi)$ is
\begin{equation}\label{eq:def_wasserstein}
W_1(\mu,\nu) := \min_{\gamma \in \HH(\mu,\nu)}
\left\{\int_{\Xi \times \Xi} d(\xi,\omega) \gamma(d\xi,d\omega) \right\},
\end{equation}
where $\HH(\mu,\nu)$ is the set of all distributions on $\Xi \times \Xi$ with marginals $\mu$ and $\nu$. The minimum in \eqref{eq:def_wasserstein} is attained because the metric $d$ is continuous \cite{gao2016wasserstein}.  
 
We consider ambiguity sets containing distributions close to the empirical distribution induced by the observed samples. Specifically, let $\Pbhat_N := \frac{1}{N}\sum^N_{i=1} \delta_{\data_i}$ be the empirical distribution constructed from samples $\{\data_i\}_{i \in [N]}$, where $\delta_{\data_i}$ is the unit point mass at $\data_i$. We define the data-driven Wasserstein ambiguity set as
\begin{equation}\label{eq:wasserstein-set}
\MM^\theta_N := \setdef{ \mu \in \PP_1(\Xi)}{ W_1(\mu,\Pbhat_N) \leq
	\theta},
\end{equation}
which contains all distributions with finite first moment that are within a distance $\theta \geq 0$ of $\Pbhat_N$. In~\cite{pichler2017quantitative}, it was shown that $\mathcal{M}^\theta_N$ is a weakly-compact subset of $\PP_1(\Xi)$.

\section{Distributionally robust risk-constrained programs and their consistency}\label{sec:drrcp}
In this section, we introduce risk-constrained programs and their distributionally robust counterparts. We consider ambiguity sets defined by the Wasserstein metric and the empirical distribution as discussed above. Our main result establishes that as the number of samples increases, the optimizers and the optimal value of the distributionally robust problems converge, in an appropriate sense, to the corresponding quantities of the original (with respect to the true data-generating distribution) risk-constrained problem. Throughout we consider $\Xi \subseteq \real^m$ and $d$ to be a complete metric. A canonical CVaR or {\it risk-constrained program} (RCP) is of the form
\begin{equation}\label{eq:def_rcp}
\begin{aligned}
\underset{x\in X}{\min} & \quad c^\intercal x
\\
\st  & \quad  \CVaR{\Pb}_{\alpha}(F(x,\xi)) \leq 0, 
\end{aligned}
\end{equation}
where $X \subseteq \Rb^n$ is a closed convex set (potentially defined via deterministic constraints), $c \in \Rb^n$, $\alpha \in (0,1)$, $\Pb \in \PP(\Xi)$ is the distribution of the uncertain parameter $\xi$ (see Section~\ref{subsec:wasserstein} for notation), and $\map{F}{\Rb^n \times \Xi}{\Rb}$ is called the {\it constraint function}. Using~\eqref{eq:CVaR-def}, we can equivalently write the RCP as
\begin{equation}\label{eq:def_equiv_rcp}
\begin{aligned}
\underset{x\in X, t \in \Rb}{\min} & \quad c^\intercal x
\\
\st  & \quad   \Eb_{\Pb}[(F(x,\xi)+t)_+] - t \alpha \leq 0.
\end{aligned}
\end{equation}	 
By equivalent, we mean that $x$ is a feasible point for~\eqref{eq:def_rcp} if and only if there exists $t$ such that $(x,t)$ is feasible for~\eqref{eq:def_equiv_rcp}.

The distributionally robust version of the RCP~\eqref{eq:def_rcp}, which we term as the {\it distributionally robust risk-constrained program} (DRRCP), is given by
\begin{equation}\label{eq:cvar-drccp}
\begin{aligned}
\underset{x \in X}{\min} & \quad c^\intercal x
\\
\st & \quad \sup_{\Qb \in \MM_N^{\theta}} \inf_{t \in \real} \Eb_{\Qb} [(F(x,\xi)+t)_+ - t\alpha] \le 0,
\end{aligned}
\end{equation}
where $\MM_N^{\theta}$ is the data-driven Wasserstein ambiguity set defined in \eqref{eq:wasserstein-set}. In other words, we require the CVaR constraint to hold for all distributions that are within a distance $\theta \geq 0$ from the empirical distribution $\widehat{\Pb}_N := \frac{1}{N} \sum^N_{i=1} \delta_{\widehat{\xi}_i}$ induced by the samples $\{\widehat{\xi}_i\}_{i \in [N]}$, drawn independently from $\Pb$. This problem is of interest when the decision-maker does not know the distribution $\Pb$ of the uncertain parameters and instead has access to samples. Thus, the optimal solution of \eqref{eq:cvar-drccp} is robust with respect to a family of distributions that are likely to have given rise to the observed samples. 

We now present a set of general assumptions.  
\begin{assumption}\longthmtitle{General assumptions on DRRCP}\label{ass:main1}
	The following hold: 
	\begin{enumerate}
		\item the function $F: \Rb^n \times \Xi \to \Rb$ is continuous,  \label{ass:m-1}
		\item for every $\xi \in \Xi$, $x \mapsto F(x,\xi)$ is convex on $X$, \label{ass:m-2}
		\item for every $x \in X$, $\xi \mapsto F(x,\xi)$ is bounded on $\Xi$, and \label{ass:m-3}
		\item $F$ is uniformly Lipschitz over the set $X$, that is, there exists $L > 0$ such that
		\begin{align*}
		\abs{F(x,\xi) - F(x,\xi')} \le L \norm{\xi - \xi'},
		\end{align*}
		for all $\xi, \xi' \in \Xi$ and all $x \in X$. \label{ass:m-4}
	\end{enumerate}
\end{assumption}

We first reformulate~\eqref{eq:cvar-drccp} into a form similar to~\eqref{eq:def_equiv_rcp}. The below result shows that the $\inf$ and the $\sup$ operators in the constraint defining DRRCP~\eqref{eq:cvar-drccp} can be interchanged. The proof is an application of the min-max theorem due to \cite{shapiro2002minimax} stated as Theorem \ref{thm:min-max-shapiro} in the appendix. The results hold under continuity, convexity, and boundedness conditions in Assumption \ref{ass:main1} and the proof is presented in the appendix.

\begin{lemma}\longthmtitle{Min-max equality for the constraint
		function}\label{lemma:drccp_minmax}
	Suppose Assumption~\ref{ass:main1}~\ref{ass:m-1}-\ref{ass:m-3} hold. Then, for every $x \in X$, we have
	\begin{align}
	\underset{\Qb \in \MM^\theta_N}{\sup} \,  \underset{t \in \Rb}{\inf} & \,
	\Eb_{\Qb} [(F(x,\xi)+t)_+ -t\alpha] \nonumber
	\\ = & \underset{t \in \Rb}{\inf} \,
	\underset{\Qb \in \MM^\theta_N}{\sup} \, \Eb_{\Qb} [(F(x,\xi)+t)_+ -t\alpha]. \label{eq:min-max-equality}
	\end{align}
\end{lemma}

As a consequence of the above result, we can write the DRRCP~\eqref{eq:cvar-drccp} equivalently as
\begin{equation}\label{eq:cvar-equiv-drccp}
\begin{aligned}
\underset{x \in X, t \in \Rb}{\min} & \quad c^\intercal x
\\
\st & \quad \sup_{\Qb \in \MM_N^{\theta}} \Eb_{\Qb} [(F(x,\xi)+t)_+ - t\alpha] \le 0.
\end{aligned}
\end{equation}
That is, $x$ is a feasible point for~\eqref{eq:cvar-drccp} if and only if there exists $t$ such that $(x,t)$ is feasible for~\eqref{eq:cvar-equiv-drccp}.  
Having reformulated the DRRCP into~\eqref{eq:cvar-equiv-drccp}, we move on to the consistency analysis. 
We require the following assumption throughout. 

\begin{assumption}\longthmtitle{Sequence of finite-sample guarantees}\label{as:fs-guarantee}
	Sequences $\setr{\beta_N} \subset (0,1)$ and $\setr{\eps_N} \subset (0,\infty)$ are such that $\sum_{N=1}^\infty \beta_N < \infty$, $\lim_{N \to \infty} \eps_N = 0$, and the following finite-sample guarantee holds for each $N \in \naturalnumbers$,
	\begin{align}\label{eq:fs-beta}
	\Pb^N (W_1(\Pb,\Pbhat_N) \le \eps_N) \ge 1-\beta_N.
	\end{align}
\end{assumption}

The above assumption imposes that as the number of samples increases, the distance between the data-generating distribution and the empirical distribution becomes vanishingly small with higher confidence. Recent works have indeed established the existence of such sequences~\cite{esfahani2018data}. We start our analysis with some preliminary lemmas. 

\begin{lemma}\longthmtitle{Uniform convergence of distributions~\cite[Lemma 3.7]{esfahani2018data}}\label{le:peyman_conv_dist}
	Under Assumption~\ref{as:fs-guarantee}, we have 
	\begin{align*}
	\Pb^\infty \Bigl( \lim_{N \to \infty} \sup_{\Qb \in \MM_N^{\eps_N}} W_1(\Qb, \Pb) = 0 \Bigr) = 1. 
	\end{align*}
\end{lemma}

The proof is analogous to the proof of~\cite[Lemma 3.7]{esfahani2018data} and is omitted in the interest of space. The above result shows that if the Wasserstein radius decreases to zero in a carefully chosen manner, then any sequence of distributions drawn from the ambiguity sets converges to the true distribution. 

\begin{remark}\longthmtitle{Comparison with~\cite{esfahani2018data}}
Following the above lemma, \cite{esfahani2018data} proves asymptotic consistency of the optimal value and optimizers of distributionally robust expected cost minimization programs under suitable boundedness and continuity assumptions on the cost function. While constrained optimization programs can be written equivalently as expected cost minimization problems via an indicator function on the feasibility set, the consistency results from \cite{esfahani2018data} do not directly apply as the indicator function is not bounded for points that violate the constraints.  \oprocend
\end{remark}

We now show that as the number of samples increases, the constraint function of the DRRCP's equivalent form~\eqref{eq:cvar-equiv-drccp} converges uniformly to that of the RCP \eqref{eq:def_equiv_rcp}. We first define
\begin{align}
v(x,t) & := \Eb_\Pb [(F(x,\xi)+t)_+ - t\alpha], \label{eq:vxt_def}
\\ 
\widehat{v}_N(x,t) & := \sup_{\Qb \in \MM_N^{\eps_N}} \Eb_{\Qb} [(F(x,\xi)+t)_+ - t\alpha], 	\label{eq:vNxt_def}
\end{align}
where note that $\widehat{v}_N$ is a random function as the ambiguity set depends on the samples. We now establish uniform $\Pb^\infty$-almost sure convergence of $\widehat{v}_N$ from above to $v$. For this, we require the constraint function to be uniformly Lipschitz continuous as stated in Assumption \ref{ass:main1}. 
\begin{lemma}\longthmtitle{Uniform convergence of $\widehat{v}_N$ from above to $v$}\label{le:unif-conv-v}
	Let Assumption~\ref{ass:main1}~\ref{ass:m-1},~\ref{ass:m-2} and~\ref{ass:m-4} hold. Further, suppose Assumption~\ref{as:fs-guarantee} holds.  Then, the following hold
	\begin{align*}
	\Pb^\infty \Bigl( v(x,t) \le \widehat{v}_N(x,t) \text{ for all sufficiently large $N$} \Bigr) & = 1,
	\\
	\Pb^\infty \Bigl( \lim_{N \to \infty} \sup_{x \in X, t \in \Rb} |\widehat{v}_N(x,t) - v(x,t)| = 0 \Bigr) & = 1,
	\end{align*}
	where the first equality is satisfied for all $(x,t) \in X \times \real$.
\end{lemma}
\begin{proof}
	Fix any $(x,t) \in X \times \real$. From~\eqref{eq:fs-beta}, we deduce that the following inequality holds with probability at least $1-\beta_N$,
	\begin{align*}
		\Eb_\Pb [(F(x,\xi) + t)_+ - t \alpha ] \le \sup_{\Qb \in \MM_N^{\eps_N}} \Eb_\Qb [(F(x,\xi) + t)_+ - t \alpha].
	\end{align*}
	That is, $\Pb^N (v(x,t) \le \widehat{v}_N(x,t) ) \ge 1-\beta_N$, for all  $N \in \naturalnumbers$. 	Since $\sum_{N=1}^\infty \beta_N < \infty$, from Borel-Cantelli Lemma~\cite[Theorem 2.3.6]{durrett2010book}, we obtain the first assertion.  
		
	From the uniform Lipschitz condition on $F$ stated in Assumption~\ref{ass:main1}~\ref{ass:m-4}, we deduce that for any fixed $(x,t) \in X \times \real$ and any $\xi, \xi' \in \Xi$, 
	\begin{align*}
	& \abs{\bigl((F(x,\xi) + t)_+ - t \alpha \bigr) - \bigl( (F(x,\xi') + t)_+ - t \alpha \bigr)}
	\\
	& \qquad \qquad \qquad= \abs{(F(x,\xi) + t)_+ - (F(x,\xi') + t)_+}
	\\
	& \qquad \qquad \qquad \le \abs{F(x,\xi) - F(x,\xi')} \le L \norm{\xi - \xi'},
	\end{align*}
	where the first inequality holds because the operator $( \cdot )_+$ is Lipschitz with constant unity. The above reasoning implies that the map $\xi \mapsto (F(x,\xi) + t)_+ - t \alpha$ is uniformly Lipschitz over the set $X \times \real$. Using this fact in the dual form of the definition of the Wasserstein metric \cite{esfahani2018data}, we conclude that
	\begin{align}
	\big| \Eb_{\Pb_1}[(F(x,\xi) + t)_+ - t \alpha]  - & \Eb_{\Pb_2}[(F(x,\xi) + t)_+ - t \alpha] \big| \notag
	\\
	& \le L W_1(\Pb_1,\Pb_2),\label{eq:lip-wass}
	\end{align}
	for any two distributions $\Pb_1$ and $\Pb_2$. Consider now a sequence of positive real numbers $\delta_N$, $N\in\Nb$ such that $\lim_{N \to \infty} \delta_N = 0$. For each $(x,t) \in X \times \Rb$, let $\Qb_N^{(x,t)} \in \MM_N^{\eps_N}$ be a $\delta_N$-optimal distribution such that
	\begin{align}\label{eq:v-ineq} 
	\Eb_{\Qb_N^{(x,t)}} & [(F(x,\xi)+t)_+ - t\alpha] \le \notag
	\\
	& \widehat{v}_N(x,t)  \leq \Eb_{\Qb_N^{(x,t)}} [(F(x,\xi)+t)_+ - t\alpha] + \delta_N.
	\end{align}
	Existence of such a distribution is due to the fact that expectation is a linear operator. Next, we have 
	\begin{align}
	|\widehat{v}_N(x,t) - v(x,t)| & \leq |\Eb_{\Qb_N^{(x,t)}} [(F(x,\xi)+t)_+ - t\alpha] \notag
	\\ 
	& \qquad - \Eb_{\Pb} [(F(x,\xi)+t)_+ - t\alpha]| + \delta_N \notag 
	\\ 
	& \leq L W_1(\Qb_N^{(x,t)},\Pb) + \delta_N \notag
	\\ 
	& \leq L \sup_{\Qb \in \MM_N^{\eps_N}} W_1(\Qb_N,\Pb) + \delta_N. \label{eq:series-q}
	\end{align}
	The first inequality above uses~\eqref{eq:v-ineq}, the second inequality follows from the condition~\eqref{eq:lip-wass}, and the last inequality due to the fact that $\Qb_N^{(x,t)} \in \MM_N^{\eps_N}$.  Since the right-hand side of~\eqref{eq:series-q} is independent of $(x,t)$, we have
	\begin{align*}
	\!\! \sup_{(x,t) \in X \times \Rb} \!\!\! |\widehat{v}_N(x,t) - v(x,t)| \leq L \!\!\sup_{\Qb \in \MM_N^{\eps_N}}  \!\!\!W_1(\Qb,\Pb) + \delta_N.
	\end{align*}
	The proof then concludes by invoking Lemma~\ref{le:peyman_conv_dist}. 
\end{proof}
We note here that the convergence from above of $\widehat{v}_N$ to $v$ is due to summability of $\beta_N$ in Assumption~\ref{as:fs-guarantee}. If one only needs convergence, then $\eps_N$ tending to zero is sufficient.
We now present our main result. We denote by $\Jrcp$ the optimal value of~\eqref{eq:def_rcp} and for a given $N$, we let $\Jdrrcp_N$ and $\{\xdrrcp_N\}_{N \in \Nb}$ denote the optimal value and an optimizer of~\eqref{eq:cvar-drccp}, resp., where $\theta$ is set to $\epsilon_N$ satisfying Assumption~\ref{as:fs-guarantee}, i.e., $\epsilon_N$ is chosen depending on $N$ and $\beta_N$ satisfying~\eqref{eq:fs-beta}.

\begin{theorem}\longthmtitle{Asymptotic consistency of the DRRCP~\eqref{eq:cvar-drccp}}\label{th:asymp-cvar}
	Let Assumptions~\ref{ass:main1} and~\ref{as:fs-guarantee} hold. 
	Assume that the feasibility set of~\eqref{eq:def_rcp} has a nonempty interior and that the optimizers of~\eqref{eq:def_rcp} belong to a compact set $\YY \subset X$. Moreover, assume that for sufficiently large $N$ and any sequence of i.i.d samples $\{\data_i\}_{i=1}^N$, optimizers of~\eqref{eq:cvar-drccp} with $\theta$ replaced with $\epsilon_N$ belong to $\YY$.   
	Then, the following statements hold $\Pb^\infty$ - almost surely:
	\begin{enumerate}
		\item $\Jrcp \le \Jdrrcp_N$ for all sufficiently large $N$,
		\item $\Jdrrcp_N \to \Jrcp$ as $N \to \infty$, and 
		\item any accumulation point of any sequence of optimizers $\{\xdrrcp_N\}_{N\in\Nb}$ is an optimal solution of the problem \eqref{eq:def_rcp}.
	\end{enumerate}
\end{theorem}
\begin{proof}
	The first statement here follows from the first assertion of Lemma~\ref{le:unif-conv-v}. For the next two statements, the proof strategy is to show an analogous convergence argument: that the optima and optimizers of~\eqref{eq:cvar-equiv-drccp} approach~\eqref{eq:def_equiv_rcp}. All convergence statements in this proof involving random quantities hold $\Pb^\infty$-almost surely and we omit restating this fact for the sake of brevity. 
	Denote the feasibility sets of~\eqref{eq:def_equiv_rcp} and~\eqref{eq:cvar-equiv-drccp} as $\FFrcp$ and $\FFdrrcp_N$, respectively. Then, $\FFrcp  = \setdef{(x,t) \in X \times \Rb}{v(x,t) \le 0}$ and $\FFdrrcp_N = \setdef{(x,t) \in X \times \Rb}{\widehat{v}_N(x,t) \le 0}$. Recall that the set $\FFdrrcp_N$ is random.
	
	\emph{Step 1: Defining $\WW$:} Since $F$ is continuous, $\YY$ is compact, and $F(x,\cdot)$ is bounded over $\Xi$ for every $x \in \YY$, we deduce that the set $\setdefb{t}{\Eb_{\Pb} [(F(x,\xi) +t)_+] - t\alpha \le 0, x \in \YY}$ is compact. Recall that optimizers of~\eqref{eq:def_rcp} belong to $\YY$. Thus, there exists a compact set $\TT \subset \real$ such that optimizers of~\eqref{eq:def_equiv_rcp} belong to the set $\WW: = \YY \times \TT$. Similarly, for all sufficiently large $N$ and all sequence of $N$ i.i.d samples, the set of optimizers of~\eqref{eq:cvar-equiv-drccp} (with $\theta$ replaced with $\epsilon_N$) belong to the set $\WW$. Since the intersection of $\YY$ and the feasibility set of~\eqref{eq:def_rcp} has a nonempty interior, one can assume, without loss of generality, that $\WW \cap \FFrcp$ has a nonempty interior. 
	
	\emph{Step 2: Establishing $\FFdrrcp_N \to \FFrcp$:} Following Lemma~\ref{le:unif-conv-v}, we know that $\widehat{v}_N$ converges uniformly $\Pb^\infty$-almost surely to $v$. Using this fact, one can establish convergence, defined in an appropriate sense, of $\FFdrrcp_N$ to $\FFrcp$.  Specifically, we will show 
	\begin{align}\label{eq:conv-dist-set}
	\lim_{N \to \infty}  \sup_{(x,t) \in \FFdrrcp_N \cap \WW } \dist((x,t), \FFrcp) = 0,
	\end{align}
	where $\dist( (x,t) , \FFrcp)$ is the distance of the point $(x,t)$ to the set $\FFrcp$, that is, $\dist( (x,t), \FFrcp) = \inf_{(x',t') \in \FFrcp} \norm{(x,t) - (x',t')}$. We proceed with a contradiction argument to show~\eqref{eq:conv-dist-set}. Recall the assertion that~\eqref{eq:conv-dist-set} holds $\Pb^\infty$-almost surely. Now, for the sake of contradiction, assume that there exists a set of sequence of i.i.d samples
	\begin{align*}
	\HH := \setdefb{ \{\data_N (\sigma) \}_{N \in \naturalnumbers}}{\sigma \in \Sigma} 
	\end{align*}
	that has finite measure under the distribution $\Pb^\infty$ and each element of $\HH$ violates the limit~\eqref{eq:conv-dist-set}. Here, $\Sigma$ is some uncountable index set.  
	To be more precise, $\HH$ gives rise to a set of sequences $\setdef{\{\FFdrrcp_N(\sigma)\}_{N \in \naturalnumbers}}{ \sigma \in \Sigma}$ such that each element in this set violates~\eqref{eq:conv-dist-set}. This in turn implies that for each $\sigma \in \Sigma$, one can assign a sequence $\setr{(x_N(\sigma),t_N(\sigma)) \in \FFdrrcp_N(\sigma) \cap \WW}_{N \in \naturalnumbers}$ and a constant $\gamma_\sigma > 0$ such that 	 
	\begin{align}\label{eq:contra-dist}
	\dist \big( (x_N(\sigma),t_N(\sigma)), \FFrcp \big) > \gamma_\sigma, \quad \forall \, N \in \naturalnumbers.
	\end{align}
	Since $\WW$ is compact, there exists a subsequence of $\{(x_N(\sigma),t_N(\sigma))\}$ that converges to some $(\xb(\sigma),\tb(\sigma)) \in \WW$. We denote this subsequence by $\{(x_N(\sigma),t_N(\sigma))\}$ for convenience. Then, due to continuity of $v$, for any $\eps/2 > 0$, there exists $N_1(\sigma) \in \naturalnumbers$ such that 
	\begin{align*}
	\abs{v(x_N(\sigma), t_N(\sigma)) - v(\xb(\sigma),\tb(\sigma))} & \le \eps/2
	\end{align*}
	for all $N \ge N_1(\sigma)$. Moreover, by $\Pb^\infty$-almost sure uniform convergence of $\widehat{v}_N \to v$, for any $\eps/2 > 0$, for almost all $\sigma \in \Sigma$, there exists $N_2(\sigma) \in \naturalnumbers$ such that 
	\begin{align*}
	\abs{\widehat{v}_N (x_N (\sigma), t_N (\sigma)) - v(x_N(\sigma), t_N(\sigma))} & \le \eps/2, 
	\end{align*}
	for all $N \ge N_2(\sigma)$. Using the above two inequalities, we conclude that for almost all $\sigma \in \Sigma$, for any $\eps > 0$, there exists $\bar{N}(\sigma)$ such that
	\begin{align*}
	\abs{\widehat{v}_N(x_N(\sigma),t_N(\sigma)) - v(\xb(\sigma),\tb(\sigma))} \le \eps, \quad \forall N \ge \bar{N}(\sigma).
	\end{align*}
	This implies that $\lim_{N \to \infty} \widehat{v}_N(x_N(\sigma),t_N(\sigma)) = v(\xb(\sigma),\tb(\sigma))$ for almost all $\sigma$. Since $\widehat{v}_N(x_N(\sigma),t_N(\sigma)) \le 0$ for all $N$, we get $v(\xb(\sigma),\tb(\sigma)) \le 0$, that is, $(\xb(\sigma),\tb(\sigma)) \in \FFrcp$ for almost all $\sigma \in \Sigma$. This is in contradiction with~\eqref{eq:contra-dist}. Hence, we have established~\eqref{eq:conv-dist-set}. 
	
	\emph{Step 3: Convergence of optimizers and optimal values:} Now let $(x_N,t_N) \in \FFdrrcpo_N$ for all $N$, where $\FFdrrcpo_N$ is the set of optimal solutions of~\eqref{eq:cvar-equiv-drccp}. Since the sequence $\setr{(x_N,t_N)}$ is contained in a compact set $\WW$, by abuse of notation, we deduce that $(x_N,t_N) \to (\xb,\tb)$ for some $(\xb,\tb) \in \WW$. Since $\FFrcp$ is closed and~\eqref{eq:conv-dist-set} holds, we get $(\xb,\tb) \in \FFrcp$. By continuity,
	\begin{align}\label{eq:first-v-ineq-1}
	\lim_{N \to \infty} c^\intercal x_N = c^\intercal \xb \ge \Jrcp,
	\end{align}	
	where $\Jrcp$ is the optimum value of~\eqref{eq:def_rcp}.
	
	Now, let $(x^*,t^*) \in \FFrcpo$, where $\FFrcpo$ is the set of optimal solutions of \eqref{eq:def_equiv_rcp}. Since $\FFrcp$ is convex and its interior is nonempty, there exists a sequence $\setr{(x_k,t_k)}_{k \in \naturalnumbers}$ belonging to the interior of $\FFrcp$ such that $(x_k,t_k) \to (x^*,t^*)$. This implies that for any $\epsilon > 0$, there exists $\bar{k}$ satisfying
	\begin{align}\label{eq:t-exp}
	c^\intercal x_{\bar{k}} - \Jrcp = c^\intercal x_{\bar{k}} - c^\intercal x^* \le \eps.
	\end{align} 
	Since $\setr{(x_k,t_k)}$ belongs to the interior of $\FFrcp$ and $\widehat{v}_N$ converges to $v$ uniformly over $X \times \Rb$, we deduce that $(x_{\bar{k}},t_{\bar{k}}) \in \FFdrrcp_N$ for all sufficiently large $N$. For such $N$, optimality of $x_N$ implies that $c^\intercal x_{\bar{k}} \ge c^\intercal x_N$. Using this fact in~\eqref{eq:t-exp}, we get	$\Jrcp \ge c^\intercal x_{\bar{k}}  - \eps \ge c^\intercal x_N - \eps$. 
	Taking $N \to \infty$ gives $\Jrcp \ge c^\intercal \xb - \eps$. Since $\eps$ can be chosen arbitrarily small, we obtain $\Jrcp \ge c^\intercal \xb$. Combined with~\eqref{eq:first-v-ineq-1}, we conclude $c^\intercal \xb = c^\intercal x^*$ and hence $\xb \in \FFrcpo$. Finally, the argument holds for any convergent subsequence of $\setr{(x_N,t_N)}$. The convergence of the optimum values then follows by continuity. 
\end{proof}

The first part of our result, that $\Jrcp \le \Jdrrcp_N$ for all sufficiently large $N$, signifies that 
the solution of the DRRCP is a conservative approximation of the solution of the RCP in the asymptotic regime. 

\begin{remark}\longthmtitle{Discussion on assumptions of Theorem~\ref{th:asymp-cvar}}
Our assumption on the interior of the feasibility set of \eqref{eq:def_rcp} being nonempty is a fairly standard assumption in consistency analysis, e.g.,~\cite[Theorem 5.5 and Proposition 5.30]{shapiro2009lectures}. This ensures that the sample-based optimization problem (problem stated in \eqref{eq:cvar-drccp}) is feasible for large $N$. A sufficient condition for this assumption to hold is the existence of  $x \in X$ such that $F(x,\xi) < 0$ for all $\xi \in \Xi$; this condition can be checked without knowing $\Pb$ or samples. 

Similarly, our assumption on the existence of a compact set $\YY \subset X$ containing the optimizers of \eqref{eq:def_rcp} and \eqref{eq:cvar-drccp} is also a standard one for consistency analysis \cite[Theorem 5.3 and Proposition 5.3]{shapiro2009lectures}, and is required to establish the convergence $\FFdrrcp_N \to \FFrcp$. It is trivially satisfied if $X$ is compact. If $X$ is unbounded, this assumption holds if $x \mapsto \Eb_{\Pb} [F(x,\xi)]$ and $x \mapsto \Eb_{\Qb}[F(x,\xi)]$ are coercive for some distribution $\Qb \in \MM^{\theta}_N$ (e.g., the empirical distribution). \oprocend
\end{remark}

Next, we analyze the consistency of distributionally robust chance-constrained programs. 

\section{Distributionally robust chance-constrained programs and their consistency}\label{sec:drccp}

Consider the \emph{chance-constrained program} (CCP), 
\begin{equation}\label{eq:def_ccp}
\begin{aligned}
\underset{x\in X}{\min} & \quad c^\intercal x
\\
\st  & \quad  \Pb( (F(x,\xi) \le 0 ) \geq 1-\alpha, 
\end{aligned}
\end{equation}
where we borrow the notation from Section~\ref{sec:drrcp}. In comparison with the RCP~\eqref{eq:def_rcp}, here, we require the uncertain constraint $F(x,\xi) \le 0$ to hold with a high probability, i.e., at least $1-\alpha$. Note that this constraint is equivalent to $\VaR{\Pb}_\alpha(F(x,\xi)) \leq 0$ and in general, the set of points satisfying the constraint is non-convex. 

The distributionally robust version of the CCP~\eqref{eq:def_ccp}, which we term as the \emph{distributionally robust chance-constrained program} (DRCCP), is given as  
\begin{equation}\label{eq:def_drccp}
\begin{aligned}
\underset{x\in X}{\min} & \quad c^\intercal x
\\
\st  & \quad  \inf_{\Qb \in \MM^\theta_N} \Qb( (F(x,\xi) \le 0 ) \geq 1-\alpha,
\end{aligned}
\end{equation}
where $\MM^\theta_N$ is the ambiguity set defined in~\eqref{eq:wasserstein-set}. We next present the consistency analysis for the DRCCP. As explained before, the chance-constraint can render the feasibility set non-convex. Therefore, consistency requires the following conditions which are different from Assumption~\ref{ass:main1}.
\begin{assumption}\label{as:regularity-ccp}
	\longthmtitle{Regularity of CCP}
	The map $F$ is uniformly continuous and either of the following holds:
	\begin{enumerate}
		\item The distribution $\Pb$ satisfies 
		\begin{align*}
		\Pb(\setdef{\xi \in \Xi}{F(x,\xi) = 0}) = 0, \quad \text{ for all } x \in X.
		\end{align*}
		\item The set-valued map 
		$H(x) := \setdef{\xi \in \Xi}{F(x,\xi) \le 0}$
		is convex-valued and continuous over $X$ (where continuity implies inner and outer semicontinuity of the set-valued map) and for any $x \in X$,
		$\Pb(\mathrm{bd} H(x)) = 0$,
		where $\mathrm{bd} H(x)$ denotes the boundary of the set $H(x)$.
	\end{enumerate}
\end{assumption}

We have the following consistency result. The proof is largely based on results from~\cite{guo2017convergence}, where the consistency analysis was done for ambiguity sets that are not random. A key step in the proof is to establish almost sure convergence of the feasibility set of the DRCCP to the feasibility set of the CCP which requires the constraint function to be continuous. Assumption \ref{as:regularity-ccp}, inspired by \cite{guo2017convergence}, states complementary sufficient conditions which ensure this; \cite[Example 4.3]{guo2017convergence} illustrates how  Assumption \ref{as:regularity-ccp} (ii) holds in cases where Assumption \ref{as:regularity-ccp} (i) fails.\footnote{The assumption is satisfied for several classes of functions. For example, if the constraint function has an affine separable form $F(x,\xi) = Ax + B\xi$, $B$ has full column rank, and $\Pb$ has a continuous distribution, then $\Pb(F(x,\xi) = 0) = 0$ for any $x$.}

\begin{theorem}\longthmtitle{Asymptotic consistency of the DRCCP~\eqref{eq:def_drccp}}\label{theorem:asymp-drccp}
	Let Assumptions~\ref{as:fs-guarantee} and~\ref{as:regularity-ccp} hold. Assume that there exists a compact set $\YY \subset X$ such that the optimizers of~\eqref{eq:def_ccp} belong to $\YY$. Suppose there exists an optimizer $x^*$ of~\eqref{eq:def_ccp} that belongs to the closure of the set
	$\setdef{x \in X}{\Pb(F(x,\xi) \le 0) > 1-\alpha}$.
	Moreover, assume that for sufficiently large $N$ and any sequence of i.i.d samples $\{\data_i\}_{i=1}^N$, optimizers of~\eqref{eq:def_drccp} with $\theta$ replaced with $\epsilon_N$ belong to $\YY$. Then, the following hold $\Pb^\infty$ - almost surely:
	\begin{enumerate}
		\item $\Jccp \le \Jdrccp_N$ for all sufficiently large $N$,
		\item $\Jdrccp_N \to \Jccp$ as $N \to \infty$, and 
		\item any accumulation point of any sequence of optimizers $\{\xdrccp_N\}_{N\in\Nb}$ is an optimizer of the problem \eqref{eq:def_ccp}.
	\end{enumerate}
	Here, $\Jccp$ is the optimal value of~\eqref{eq:def_ccp} and for a given $N$, $\Jdrccp_N$ and $\{\xdrccp_N\}_{N \in \Nb}$ are the optimal value and an optimizer of~\eqref{eq:def_drccp}, respectively, where $\theta$ is set to $\epsilon_N$. 
\end{theorem}
\begin{proof}
	By assumption, without loss of generality, one can assume that $X = \YY$ is a compact set. Define 
	\begin{align}
		\vccp(x) & := \Pb (F(x,\xi) \le 0), \label{eq:vccp_def}
		\\ 
		\vdrccp(x) & := \inf_{\Qb \in \MM_N^{\eps_N}} \Qb(F(x,\xi) \le 0), 	\label{eq:vdrccp_def}
	\end{align}
	where $\{\eps_N\}_{N\in\naturalnumbers} \subset (0,\infty)$ is any sequence satisfying the hypotheses. Using Assumption~\ref{as:fs-guarantee} and following a similar reasoning as presented in the proof of Lemma~\ref{le:unif-conv-v}, we conclude that for any $x \in X$,
	\begin{align*}
		\Pb^\infty \Bigl(\vdrccp(x) \le \vccp(x) \text{ for all sufficiently large $N$} \Bigr) = 1.
	\end{align*}
	Consequently, $\Jccp \le \Jdrccp_N$ for all sufficiently large $N$. Regarding the convergence statements, note that from~\cite[Theorem 4.9]{guo2017convergence}, Assumption~\ref{as:regularity-ccp} implies continuity of $\vccp$. Further,  from Lemma~\ref{le:peyman_conv_dist}, we deduce that $\MM_N^{\eps_N}$ converges weakly to $\Pb$ almost surely. That is, almost surely, any sequence $\setr{\Pb_N \in \MM_N^{\eps_N}}$ converges weakly to $\Pb$. Thus, from~\cite[Proposition 5.2, 5.3 and Theorem 3.2]{guo2017convergence}, we obtain 	
		$\Pb^\infty \Bigl( \lim_{N \to \infty} \sup_{x \in X} |\vdrccp(x) - \vccp(x)| = 0 \Bigr) = 1$. 
	The proof concludes by applying~\cite[Theorem 3.4]{guo2017convergence}. 
\end{proof}

\section{Conclusion}\label{sec:conclusions}
We have studied the asymptotic consistency of data-driven distributionally robust risk- (captured by the CVaR) and chance-constrained optimization under Wasserstein ambiguity sets. As a consequence, under suitable assumptions on the problem data, the distributionally robust versions of the problems can be used as ``robust approximators'' of the original problems. In future, we plan to analyze the rate of convergence of the consistency arguments. Particularly, we wish to 
obtain confidence intervals for original optimizers of the problems using the solutions of the distributionally robust counterparts.

\section*{Appendix}
\renewcommand{\theequation}{A.\arabic{equation}}
\renewcommand{\thetheorem}{A.\arabic{theorem}}

The following result aids us in proving Lemma \ref{lemma:drccp_minmax}.

\begin{theorem}\longthmtitle{Stochastic min-max equality \cite{shapiro2002minimax}}\label{thm:min-max-shapiro}
	Let $\MM$ be a nonempty and weakly compact set of probability measures on $(\Xi,\mathcal{B}(\Xi))$.
	Consider a function $g:\Rb^n \times \Xi \to \Rb$. Let $T \subseteq \Rb^n$ be a closed convex set. Assume that there exists a convex neighborhood $V$ of $T$ such that for all $t \in V$, the function $g(t,\cdot)$ is measurable, integrable with respect to all $\Pb \in \MM$, and $\sup_{\Pb \in \MM} \Eb_{\Pb} [g(t,\xi)] < \infty$. Further assume that $g(\cdot,\xi)$ is convex on $V$ for all $\xi \in \Xi$. Let $\bar{t} \in \argmin_{t \in T} \sup_{\Pb \in \mathcal{M}} \mathbb{E}_{\Pb}[g(t,\xi)]$. Assume that for every $t$ in a neighborhood of $\bar{t}$, the function $g(t,\cdot)$ is bounded and upper semicontinuous on $\Xi$ and the function $g(\bar{t},\cdot)$ is bounded and continuous on $\Xi$. Then,
	\begin{equation*}
	\inf_{t \in T} \sup_{\Pb \in \MM} \Eb_{\Pb}[g(t,\xi)] = \sup_{\Pb \in \MM} \inf_{t \in T} \Eb_{\Pb}[g(t,\xi)].
	\end{equation*}
\end{theorem}

\medskip

\noindent {\bf Proof of Lemma \ref{lemma:drccp_minmax}:} We suppress the variable $x$ in the proof for better readability.  We verify that the hypotheses of the min-max theorem (Theorem \ref{thm:min-max-shapiro}) hold.
	
	Drawing the parallelism in notation between our case and Theorem~\ref{thm:min-max-shapiro}, note that here $\Rb$ plays the role of both $T$ and $V$; $\MM^\theta_N$ that of $\MM$; and the function $g$ is $g(t,\xi):=(F(\xi)+t)_+ - t\alpha$. Recall that $\MM^\theta_N$ is weakly compact.

	Note that $g$ is continuous as $F$ is so. Further since $F$ is bounded, for every $t \in \Rb$, the function $\xi \mapsto g(t,\xi)$ is bounded and $\sup_{\Qb \in \MM^\theta_N} \Eb_{\Qb} [g(t,\xi)] < \infty$. Finally, for every $\xi \in \Xi$, $t \mapsto g(t,\xi)$ is convex. Thus, to conclude the proof it remains to show that the infimum on the right-hand side of~\eqref{eq:min-max-equality} is attained. To this end, define the function
	$h(t):= \underset{\Qb \in \MM^\theta_N}{\sup} \Eb_{\Qb} [(F(\xi)+t)_+ -t\alpha]$.
Now, for any $\Qb \in \MM^\theta_N$, the function $t \mapsto \Eb_{\Qb} [(F(\xi)+t)_+ - t \alpha]$ is convex and real-valued. Since $h$ is supremum over a family of such functions, $h$ is convex, lower semicontinuous~\cite[Proposition 2.1.2]{JBHU-CL:04}. Further, for any $\xi$, $(F(\xi)+t)_+ - t \alpha \to \infty$ as $\abs{t} \to \infty$. This fact along with boundedness of $F$ implies $h(t) \to \infty$ as $\abs{t} \to \infty$. Thus, $\inf_{t \in \Rb} h(t)$ exists. \hfill $\blacksquare$

\end{document}